\documentclass{amsart}

\usepackage{graphicx}
\usepackage[all]{xy}
\usepackage{enumitem}
\usepackage{color}
\usepackage{hyperref}
\usepackage{amssymb}
\usepackage{tikz-cd}
\usepackage{mathtools}
\usepackage{etoolbox}
\apptocmd{\sloppy}{\hbadness 10000\relax}{}{}
\usepackage{kantlipsum} 

\calclayout

\theoremstyle{definition}
\newtheorem{theorem}{Theorem}[section]

\newtheorem{lemma}[theorem]{Lemma}

\newtheorem{conjecture}[theorem]{Conjecture}
\newtheorem*{conjecture*}{Conjecture}

\newtheorem*{question*}{Question}

\theoremstyle{remark}

\newtheorem*{remark*}{Remark}
\newtheorem*{remarks*}{Remarks}

\newtheorem*{claim*}{Claim}

\numberwithin{equation}{section}

\newcommand{\Z}{\mathbb{Z}}
\newcommand{\Q}{\mathbb{Q}}
\newcommand{\R}{\mathbb{R}}

\newcommand{\id}{\mathrm{id}}

\newcommand{\inn}{\mathrm{int}}

\newcommand{\pia}{\pi_{1}(Y_{1})}
\newcommand{\pib}{\pi_{1}(Y_{2})}
\newcommand{\pii}{\pi_{1}(Y_{i})}
\newcommand{\piw}{\pi_{1}(W)}
\newcommand{\piwa}{\pi_{1}(W_{1})}
\newcommand{\piwb}{\pi_{1}(W_{2})}
\newcommand{\piwi}{\pi_{1}(W_{i})}
\newcommand{\pibarw}{\pi_{1}(\barw)}

\newcommand{\barw}{\overline{W}}

\newcommand{\rna}{R_{N}(Y_{1})}
\newcommand{\rnb}{R_{N}(Y_{2})}
\newcommand{\rnw}{R_{N}(W)}
\newcommand{\son}{\mathrm{SO}(n)}

\begin{document}

\title{Ribbon cobordisms as a partial order}

\author{Marius Huber}
\address{Department of Mathematics, Boston College, Chestnut Hill, MA 02467}
\email{marius.huber@bc.edu}






\begin{abstract}
We show that the notion of ribbon rational homology cobordism yields a partial order on the set of aspherical $3$-manifolds, thus supporting a conjecture formulated by Daemi, Lidman, Vela-Vick and Wong.
Our proof is built on Agol's recent proof of the fact that ribbon concordance yields a partial order on the set of knots in the 3-sphere.
\end{abstract}

\maketitle


\section{Introduction}\label{sec_intro}

In a recent preprint \cite{Agol}, Agol famously proved that ribbon concordance  defines a partial order on the set of knots in $S^{3}$, thus resolving a long-standing conjecture \cite[Conjecture 1.1]{Gordon} due to Gordon.
Agol's proof relies on relating maps between knot groups to maps between the representation varieties of those groups.
Before Agol's result, work by Zemke \cite{Zemke} sparked renewed interest in the study of ribbon concordances (see e.g. \cite{Levine-Zemke}, \cite{Juhasz-Miller-Zemke} and \cite{Sarkar}), which was eventually put into a broader framework in an article by Daemi, Lidman, Vela-Vick and Wong \cite{DLVVW} in which they defined and studied ribbon rational homology cobordisms.

Recall that, given closed, oriented $3$-manifolds $Y_{1}$ and $Y_{2}$, a rational homology cobordism from $Y_{1}$ to $Y_{2}$ is an oriented $4$-manifold $W$ with oriented boundary $\partial W=-Y_{1}\amalg Y_{2}$ such that inclusion of $Y_{i}$ into $W$ induces an isomorphism $H_{*}(Y_{i};\Q)\cong H_{*}(W;\Q)$, $i=1,2$.
Such a cobordism $W$ from $Y_{1}$ to $Y_{2}$ is said to be ribbon, if $W$ admits a handle decomposition relative to $Y_{1}\times I$ that consists of just $1$- and $2$-handles, where $I=[0,1]$ (for the corresponding notion for $3$-manifolds with boundary, see ``Conventions and notation" at the end of this section).
In the following, we refer to ribbon rational homology cobordisms simply as \emph{ribbon cobordisms}.

This terminology stems from the fact that the exterior of a ribbon concordance $C\subset S^{3}\times I$ from a knot $K_{1}$ to $K_{2}$ admits a natural handle decomposition which makes $(S^{3}\times I)\setminus\nu C$ into a ribbon cobordism from $S^{3}\setminus\nu K_{1}$ to $S^{3}\setminus\nu K_{2}$ in the above sense.
Using this point of view, Daemi, Lidman, Vela-Vick and Wong formulated a conjecture analogous to that of Gordon.

\begin{conjecture}[\text{Daemi-Lidman-Vela-Vick-Wong \cite[Conjecture 1.1]{DLVVW}}]\label{conj_DLVVW}
The preorder on the set of homeomorphism classes of closed, connected, oriented $3$-manifolds given by ribbon cobordisms is a partial order.
\end{conjecture}

As speculated by Agol \cite{Agol}, it is natural to wonder whether the techniques used to prove Gordon's Conjecture could be used to make progress in proving Conjecture \ref{conj_DLVVW}.
The purpose of the present note is to do exactly that, and to show that Conjecture \ref{conj_DLVVW} holds for the class of aspherical $3$-manifolds (recall that a $3$-manifold $Y$ is aspherical if $\pi_{k}(Y)=0$ for all $k\geq 2$, or, equivalently, if $Y$ is irreducible and has infinite fundamental group).
Combining this with previous work of the author \cite{Huber} on ribbon cobordisms between lens spaces, we obtain the following.
Note that we do not require the manifolds involved to be closed.

\begin{theorem}\label{thm_conj}
Let $Y_{1}$ and $Y_{2}$ be compact, oriented, $3$-manifolds, possibly with boundary, such that there exists a ribbon cobordism $W_{i}$ from $Y_{i}$ to $Y_{j}$, $\{i,j\}=\{1,2\}$.
If $Y_{i}$ is either aspherical or a lens space, $i=1,2$, then $Y_{1}\cong Y_{2}$.
\end{theorem}

The above result is a rather direct consequence of the following two more technical results.
To put these into context, recall that if $Y_{1}$ and $Y_{2}$ are compact $3$-manifolds and  if $W$ is a ribbon cobordism from $Y_{1}$ to $Y_{2}$, one has the following diagram of maps induced by inclusion (\cite[Theorem 1.14]{DLVVW} and \cite[Lemma 3.1]{Gordon}):

\begin{equation*}
\begin{tikzcd}
\pia \arrow[hook]{r} & \piw & \pib \arrow[swap, two heads]{l}
\end{tikzcd}
\end{equation*}

Hence, if $Y$ is a $3$-manifold with finite fundamental group, and if $W$ is a ribbon cobordism from $Y$ to itself, then the inclusion of either boundary component into $W$ induces an isomorphism of fundamental groups.
Our main technical result states that this remains true if $Y$ has infinite fundamental group, provided that $Y$ is aspherical.

\begin{theorem}\label{mainthm}
Let $Y$ be a compact, oriented, aspherical $3$-manifold, possibly with boundary, and suppose that $W$ is a ribbon cobordism from $Y_{1}$ to $Y_{2}$, where $Y_{i}\cong Y$, $i=1,2$.
Then the inclusion of $Y_{i}$ into $W$ induces an isomorphism $\pii\cong\piw$, $i=1,2$.
\end{theorem}

As a consequence, we obtain the following result concerning pairs of $3$-manifolds with the property that there exists a ribbon cobordism in either direction.

\begin{theorem}\label{mainthm_2}
Let $Y_{1}$, $Y_{2}$ be compact, oriented, aspherical $3$-manifolds, $i=1,2$, possibly with boundary, and suppose that there exists a ribbon cobordism $W_{i}$ from $Y_{i}$ to $Y_{j}$, $\{i,j\}=\{1,2\}$.
Then the inclusion of $Y_{i}$ into $W_{j}$ induces an isomorphism $\pii\cong\pi_{1}(W_{j})$, $i,j=1,2$.
In particular, there exists an orientation-preserving homotopy equivalence $f\colon(Y_{i},\partial Y_{i})\to(Y_{j},\partial Y_{j})$, $\{i,j\}=\{1,2\}$.
\end{theorem}

We conclude this introduction by pointing out that during the time of writing of the present note, essentially the same result was independently found by Friedl, Misev and Zentner \cite{FMZ}.

\subsection*{Conventions and notation}

Given manifolds $Y$ and $Y'$, $Y\cong Y'$ means that $Y$ and $Y'$ are related by an orientation-preserving homeomorphism.
By a handle decomposition of a $4$-dimensional cobordism $W$ from $Y$ to $Y'$, we mean that $W$ is built from $Y\times[0,1]$ by attaching $1$-, $2$- and $3$-handles, where the attaching region of each handle is supported in $\inn(Y)\times\{1\}$, or in the boundary of previously attached handles.
In particular, if $\partial Y$ is non-empty, the attaching regions of the handles of $W$ avoid $\partial Y\times\{1\}\subset Y\times\{1\}$, and $\partial Y\cong\partial Y'$.

\subsection*{Acknowledgments}

I would like to thank to my advisor, Josh Greene, for the many helpful discussions we had along the course of this project.


\section{Proofs of results}\label{sec_proofs_mainthm}

The following lemma is used in the algebro-geometric portion of the proof of Theorem \ref{mainthm}, which, in turn, is virtually the same as the proof of \cite[Theorem 1.2]{Agol}.
We provide a proof of the lemma for completeness, but also to highlight the use of residual finiteness of fundamental groups of $3$-manifolds.

\begin{lemma}\label{lemma_res_fin}
Suppose $\Gamma$ is a residually finite group, and let $\gamma\in\Gamma\setminus\{1\}$.
Then there exists $n>0$ and a homomorphism $\rho\colon\Gamma\to\son$ such that $\rho(\gamma)\neq 1$.
\end{lemma}
\begin{proof}
We first show that any finite group embeds into $\son$ for some $n>0$.
For this, recall that the symmetric group on $n$ elements $S_{n}$ is generated by the $n-1$ transpositions $\tau_{i,i+1}=(i,i+1)$, $i=1,\dots,n-1$.
For $i=1,\dots,n-1$, define $\varphi_{(i,i+1)}\colon\R^{n}\to\R^{n}$ by \[\varphi_{(i,i+1)}(x_{0},\dots,x_{i},x_{i+1},\dots,x_{n})=(-x_{0},\dots,x_{i+1},x_{i},\dots,x_{n}).\]
One can check that $\varphi_{(i,i+1)}\in\mathrm{SO}(n+1)$ for all $i=1,\dots,n-1$, and hence the above assignment defines an embedding of $S_{n}$ into $\mathrm{SO}(n+1)$.
Since any finite group embeds into $S_{n}$ for some $n>0$, it follows that the same holds with $S_{n}$ replaced by $\son$.

Now, let $\Gamma$ be residually finite and $\gamma\in\Gamma$ non-trivial.
By definition of residual finiteness, there exists a finite group $G$ and a surjection $q_{\gamma}\colon\Gamma\to G$ such that $q_{\gamma}(\gamma)\neq 1$.
Postcomposing $q_{\gamma}$ with an embedding of $G$ into $\son$, for some $n>0$, yields the claim.
\end{proof}

\begin{proof}[Proof of Theorem \ref{mainthm}]
Suppose that $W$ is a ribbon cobordism as in the statement of Theorem \ref{mainthm}.
By \cite[Theorem 1.14]{DLVVW}, we have that
\begin{equation}\label{diag_inclusions}
\begin{tikzcd}
\pia \arrow[hook]{r}{i_{1}} & \piw & \pib \arrow[swap, two heads]{l}{i_{2}},
\end{tikzcd}
\end{equation}
where $i_{k}$ is the map induced by inclusion $\iota_{k}\colon Y_{k}\to W$, $k=1,2$, and $\pia\cong\pib$.
As in \cite{Agol}, for $n>0$ and a manifold $X$, let $R_{n}(X)=R_{n}(\pi_{1}(X))$ denote the representation variety of $\pi_{1}(X)$ to $\son$.
By \cite[Proposition 2.1]{DLVVW}, we then have that
\begin{equation}\label{diag_restrictions}
\begin{tikzcd}
\rna & \arrow[swap, two heads]{l}{r_{1}} \rnw \arrow[hook]{r}{r_{2}} & \rnb,
\end{tikzcd}
\end{equation}
where $r_{k}$ is the restriction map, $k=1,2$.
As shown in \cite{Agol}, $r_{1}$ is obtained by projection of $\rnw$ onto the subspace spanned by the coordinates corresponding to $\pia$ (regarded as a subgroup of $\piw$) and hence is a polynomial map, and, moreover, $\rna$ and $\rnb$ are related by a polynomial isomorphism.
Precomposing this isomorphism with $r_{1}$, one obtains a surjective polynomial map $\varphi\colon\rnw\to\rnb$.
By the argument given in \cite{Agol}, $r_{2}$, in fact, embeds $\rnw$ into $\rnb$ as a real algebraic subset.
This allows one to show that $i_{2}\colon\pib\to\piw$ is injective as follows.
Given $\gamma\in\pib\setminus\{1\}$, one can, using residual finiteness of $3$-manifold groups (which follows from \cite[Theorem 3.3]{Thurston} and Geometrization) and Lemma \ref{lemma_res_fin}, find $n>0$ and a representation $\rho\in\rnb$ with the property that $\rho(\gamma)\neq 1$.
By the above, $\rnw\subset\rnb$ is an algebraic subset that admits a surjective polynomial map to $\rnb$ (namely, $\varphi$).
Thus, \cite[Lemma A.2]{Agol} implies that $\rnw=\rnb$, and it follows that the representation $\rho$ is the restriction of some representation $\rho'\in\rnw$.
Hence, $\rho'(i_{2}(\gamma))=(r_{2}(\rho'))(\gamma)=\rho(\gamma)\neq 1$, which implies that $i_{2}(\gamma)$ is non-trivial.
It follows that $i_{2}$ is injective and hence an isomorphism.

It remains to show that $i_{1}$ is an isomorphism.
To show this, we adapt an argument used in the proof of \cite[Proposition 9.2]{DLVVW}.
Set $\barw=-W$, so that $\barw$ is a rational homology cobordism from $Y_{2}$ to $Y_{1}$ which is built from $Y_{2}\times I$ by attaching $2$- and $3$-handles.
By what we have shown so far, the map $\pib\to\pibarw$ induced by inclusion is an isomorphism, which implies that each of the $2$-handles of $\barw$ is attached to $Y_{2}\times I$ along a null-homotopic curve in $Y_{2}\times\{1\}$, and hence each attaching curve bounds an immersed disk in $Y_{2}\times\{1\}$.
Let $\barw_{2}$ denote the space obtained by attaching just the $2$-handles of $\barw$ to $Y_{2}\times I$.
Define a map $\rho_{2}\colon\barw_{2}\to\ Y_{2}$ as follows.
First, shrink each of the $2$-handles to its core, then map each core to the disk in $Y_{2}\times\{1\}$ bounded by its attaching curve via a map that is the identity on the attaching curve itself, and, finally, apply the obvious deformation retraction of $Y_{2}\times I$ onto $Y_{2}=Y_{2}\times \{0\}$.
The obstruction to extending $\rho_{2}$ over the $3$-handles of $\barw$ lies in $H^{3}(\barw,\barw_{2};\pi_{2}(Y_{2}))$ (see e.g. \cite[Proposition 4.72]{Hatcher}).
Since we assumed $Y_{2}$ to be aspherical, this group vanishes, and $\rho_{2}$ extends to a retraction $\rho\colon\barw\to\ Y_{2}$.
Indeed, by definition of a ribbon cobordism between manifolds with boundary, we have that $\partial Y_{1}=\partial Y_{2}\times\{1\}\subset\barw$, and it follows that $\rho$ is the identity in a neighborhood of $\partial Y_{2}\times\{1\}\subset\barw$.
Letting $\rho_{*}\colon\pi_{1}(\barw)\to\pib$ denote the map induced by $\rho$, it follows that $\rho_{*}\circ i_{2}=(\rho\circ\iota_{2})_{*}=(\id_{Y_{2}})_{*}=\id_{\pib}$, which implies that $\rho_{*}$ is an isomorphism, because $i_{2}$ is.
Consider now the map $f=\rho\circ \iota_{1}\colon Y_{1}\to Y_{2}$.
Since $H_{3}(W, Y_{2}\times[0,1];\Z)\cong H_{3}(Y_{2}, \partial Y_{2};\Z)\cong\Z$, and because $\rho$ is a retraction, $\rho$ induces an isomorphism on the level of third integral homology.
Similarly, the inclusion $\iota_{1}\colon Y_{1}\to W$ induces an isomorphism of third integral homology groups.
It follows that the map induced by $f$ sends the relative fundamental class $[Y_{1},\partial Y_{1}]\in H_{3}(Y_{1},\partial Y_{1};\Z)$ to $[Y_{2},\partial Y_{2}]\in H_{3}(Y_{2},\partial Y_{2};\Z)$.
That is, $f$ is an orientation-preserving degree one map.
Now, by \eqref{diag_inclusions}, $i_{1}$ is injective, and it follows that $f_{*}\colon\pia\to\pib$ is injective, because $\rho_{*}$ is an isomorphism.
By \cite[Lemma 1.2]{Rong}, $f_{*}$ is also surjective, and hence an isomorphism.
It follows that $i_{1}\colon\pia\to\piw$ is an isomorphism, as desired.
\end{proof}

\begin{proof}[Proof of Theorem \ref{mainthm_2}]
Let $W$ be the composition of the cobordisms $W_{1}$ and $W_{2}$, i.e. $W=W_{1}\cup_{Y_{2}}W_{2}$, so that $W$ is a ribbon cobordism from $Y_{1}$ to itself.
Letting $h_{i}\colon\piwi\to\piw$ denote the map induced by inclusion of $W_{i}$ into $W$, $i=1,2$, and using \cite[Theorem 1.14]{DLVVW}, we obtain the following diagram of maps induced by inclusion on the level of fundamental groups.

\begin{equation}\label{eq_longchain}
\begin{tikzcd}
\pia \arrow[r, "i_{1}^{1}", hook] \arrow[rrd, bend right=13, "\cong", sloped] & \piwa \arrow[rd, "h_{1}"] & \pib \arrow[l, "i_{2}^{1}"', two heads] \arrow[r, "i_{1}^{2}", hook] & \piwb \arrow[ld, "h_{2}"'] & \pia \arrow[l, "i_{2}^{2}"', two heads] \arrow[swap, lld, bend left=13, "\cong"', sloped] \\ & & \piw & &
\end{tikzcd}
\end{equation}

The fact that the maps $h_{1}\circ i_{1}^{1}$ and $ h_{2}\circ i_{2}^{2}$ are isomorphisms follows from Theorem \ref{mainthm}.
This immediately implies that $i_{2}^{2}$ is injective, and hence an isomorphism.
Switching the roles of $W_{1}$ and $W_{2}$, we see that $i_{2}^{1}$ is an isomorphism as well.
Moreover, using the fact that $W_{2}$ is a ribbon cobordism, it follows by an argument similar to the one used to show injectivity of $i_{1}^{1}$ and $i_{1}^{2}$ (see e.g. the proof of \cite[Proposition 2.1]{DLVVW}) that $h_{1}$ is injective.
Note that, for that argument to apply, we need the fact that $\piwa$ is residually finite; but $\piwa\cong\pib$ via $i_{2}^{1}$, and $\pib$, being the fundamental group of a compact $3$-manifold, is residually finite.
Now, since $h_{1}\circ i_{1}^{1}$ is an isomorphism, $h_{1}$ is also surjective and hence an isomorphism, which implies that $i_{1}^{1}$ is an isomorphism, too.
Switching the roles of $W_{1}$ and $W_{2}$, we see that $h_{2}$ and $i_{2}^{1}$ are isomorphisms as well.
Hence all maps in \eqref{eq_longchain} are isomorphisms.

It remains to prove the existence of the claimed homotopy equivalences.
To that end, note that, because all horizontal maps in \eqref{eq_longchain} are isomorphisms, we can apply the argument from the last paragraph of the proof of Theorem \ref{mainthm} to $W_{1}$ to obtain an orientation-preserving degree one map $f\colon Y_{1}\to Y_{2}$ that induces an isomorphism of fundamental groups.
Since we assumed $Y_{1}$ and $Y_{2}$ to be aspherical, it follows that $f$ induces an isomorphism on all homotopy groups and hence is a homotopy equivalence by Whitehead's theorem (see e.g. \cite[Theorem 4.5]{Hatcher}).
Moreover, by construction of the retraction $\rho$ from the proof of Theorem \ref{mainthm}, we have that $f(\partial Y_{1})\subset\partial Y_{2}$ and, indeed, that $f$ fixes $\partial Y_{1}$ pointwise.
It follows that $f\colon(Y_{1},\partial Y_{1})\to(Y_{2},\partial Y_{2})$ is a homotopy equivalence, as desired.
The above argument applied to $W=W_{2}$ yields a homotopy equivalence going in the other direction.
\end{proof}

We are now in a position to prove our main result.

\begin{proof}[Proof of Theorem \ref{thm_conj}]
Observe that, by \eqref{eq_longchain}, $Y_{1}$ is a lens space iff $Y_{2}$ is.
Assume first that $Y_{1}$, and hence, by definition of ribbon cobordism, also $Y_{2}$, is closed.
If both $Y_{1}$ and $Y_{2}$ are lens spaces, the claim is a straightforward consequence of \cite[Theorem 1.2]{Huber}, so we may assume that both $Y_{1}$ and $Y_{2}$ are aspherical.
By Theorem \ref{mainthm_2}, there exists an orientation-preserving homotopy equivalence $f\colon Y_{1}\to Y_{2}$.
Note that $\pib$ is not just infinite, but torsion-free by \cite[(C.2) and (C.3)]{Aschenbrenner-Friedl-Wilton}, and hence the Borel Conjecture in dimension three \cite[Theorem 0.7]{Kreck-Lueck} implies that $f$ is homotopic to a homeomorphism.
This homeomorphism must be orientation-preserving, because this property is preserved under homotopy, and it follows that $Y_{1}\cong Y_{2}$.

It remains to address the case where $Y_{1}$, and hence also $Y_{2}$, has non-empty boundary.
By what we assumed, it follows that both $Y_{1}$ and $Y_{2}$ are aspherical.
By Theorem \ref{mainthm_2}, there exists an orientation-preserving homotopy equivalence $f\colon(Y_{1},\partial Y_{1})\to(Y_{2},\partial Y_{2})$.
Since $Y_{i}$ has non-empty boundary, and hence is Haken, $i=1,2$, $f$ is homotopic to a homeomorphism from $Y_{1}$ to $Y_{2}$ by \cite[Corollary 6.5]{Waldhausen}.
As before, this homeomorphism must be orientation-preserving, because $f$ was, and it follows that $Y_{1}\cong Y_{2}$.
\end{proof}


\bibliographystyle{alpha}
\bibliography{biblio.bib}

\end{document}